\documentclass[a4paper,12pt]{amsart}
\usepackage[cm]{fullpage}
\usepackage[pagebackref]{hyperref}

\usepackage{todonotes}
\setuptodonotes{inline}

\usepackage{mymacros,oitp}

\title{Marked Artin--Schelter surfaces of del Pezzo types}
\author{Shinnosuke Okawa}
\address{
Department of Mathematics,
Graduate School of Science,
Osaka University,
Machikaneyama 1-1,
Toyonaka,
Osaka,
560-0043,
Japan.
}
\email{okawa@math.sci.osaka-u.ac.jp}

\author{Kazushi Ueda}
\address{
Graduate School of Mathematical Sciences,
The University of Tokyo,
3-8-1 Komaba,
Meguro-ku,
Tokyo,
153-8914,
Japan.}
\email{kazushi@ms.u-tokyo.ac.jp}
\date{}
\begin{document}

%
%

\begin{abstract}
We introduce a class of noncommutative surfaces
called Artin--Schelter surfaces of del Pezzo types,
which contains del Pezzo surfaces as special cases.
We show that
the moduli stacks of marked Artin--Schelter surfaces of del Pezzo types
are birational to the moduli stacks of tuples
consisting of a smooth projective curve of genus 1,
two line bundles of degree 3,
and a collection of line bundles of degree 1 on the curve.
\end{abstract}

\maketitle

\section{Introduction}

A \emph{del Pezzo surface}
is a smooth projective surface $X$
whose anti-canonical divisor $-K_X$ is ample.
A \emph{marking} of $X$
is a choice of an isomorphism
from a fixed lattice $\Lambda$
to the Picard lattice $\Pic X$,
sending a particular element $\omega \in \Lambda$
to the canonical bundle.
If the base field is algebraically closed,
then
$X$
is either $\bP^1 \times \bP^1$
or $\bP^2$ blown up at
$
n = 9 - K_X^2
$ 
points in general position,
and the moduli space of marked del Pezzo surfaces of degree
$
d = K_X^2 \le 5
$
can naturally be identified
with a subspace of the configuration space
$X(3,n) \coloneqq
\left(
    (
    \bP^2
)
^{ n }
\setminus \Delta^{\mathrm{big}}
\right)
/\Aut \bP^2$
of $n$ distinct points on $\bP^2$.

A \emph{noncommutative projective plane}
is a category of the form
$
\qgr S = \gr S / \tor S
$
for an \emph{Artin--Schelter regular algebra} (\emph{AS-regular algebra} for short) $S$
of dimension 3
whose Hilbert polynomial coincides
with that of the polynomial ring
\cite{MR917738}.
While such an AS-regular algebra
is classified by a triple $(E, \sigma, L)$
consisting of a cubic curve $E$,
an automorphism $\sigma$ of $E$,
and a very ample line bundle $L$ of degree 3 on $E$
satisfying a certain generality condition,
a noncommutative projective plane is classified
by the triple
\(
\left(
E,
L_0 \! \coloneqq L,
L_1 \! \coloneqq \sigma^{ \ast } L
\right)
\)
whose isomorphism class is generically in 9-to-1 correspondence
with that of $(E, \sigma, L)$
\cite{MR1086882,MR1230966}.
Here 9
is the number of 3-torsion
translations of the curve $E$.
This 9-to-1
correspondence is a consequence
of the distinction
between AS-regular algebras and AS-regular $\bZ$-algebras:
Given an AS-regular algebra $S$,
the corresponding AS-regular $\bZ$-algebra $\Sv$
is naturally isomorphic to the full subcategory of $\qgr S$
consisting of $( S(i) )_{i \in \bZ}$,
and if the AS-regular algebra $S$ is general enough,
then there are exactly nine isomorphism classes
of AS-regular algebras
such that the corresponding AS-regular $\bZ$-algebras are isomorphic.
The isomorphism class of the
\(
   \bZ
\)-algebra
\(
   \Sv
\)
is determined by the equivalence class of the category
\(
   \qgr S
   \simeq
   \qgr \Sv
\)
\cite[Theorem~11.2.3 and Corollary~11.2.4]{MR1816070}.

The sequence $( S(i) )_{i \in \bZ}$ of objects of $\qgr S$
is an example of a \emph{very strong helix}
in the sense of \pref{df:very strong helix}.
Any del Pezzo surface $X$
admits a very strong helix
$
(E_i)_{i \in \bZ}
$
of locally free sheaves,
so that $\coh X$ is equivalent to the $\qgr$
of the corresponding AS-regular $\bZ$-algebra.
The rolled-up helix algebra
of a very strong helix on a del Pezzo surface
is 
a Calabi--Yau algebra of dimension 3
\cite{0612139}
and hence described
by a quiver $(Q,\Phi)$ with potential
\cite{MR2355031,MR3338683}.
With this in mind,
we introduce the following terminology.

\begin{definition}\label{df:main}
\begin{enumerate}
\item
A category
\(
   \cC
\)
is an \emph{Artin--Schelter surface}
if there exists an AS-regular $\bZ$-algebra $A$
of dimension 3 such that
\(
   \cC
\)
is equivalent to the category $\qgr A$.
\item
An Artin--Schelter surface
\(
    \cC
\)
is \emph{of del Pezzo type}
if there exists a quiver $Q$
which is the rolled-up helix quiver
of a very strong helix
on a del Pezzo surface and
an AS-regular $\bZ$-algebra
\(
   A
\)
of type $Q$ in the sense of \pref{df:AS-regular of type Q}
such that
\(
    \cC
\)
is equivalent to $\qgr A$.
\item
A \emph{marking} of an Artin--Schelter surface $\cC$
is a choice of a very strong helix in $D^b \cC$
such that $\qgr A$ for the associated AS-regular $\bZ$-algebra $A$
is equivalent to $\cC$.
The quiver describing the rolled-up helix algebra
is called the \emph{type} of the marking.
\item
Let $Q$ be the quiver
associated with a very strong helix
on a del Pezzo surface.
The \emph{moduli stack of marked Artin--Schelter surfaces of type $Q$}
is the open substack
of the moduli stack $\cM_Q$ of potentials
of the quiver $Q$
in the sense of \pref{df:moduli stack of potentials}
parametrizing potentials whose associated \(\bZ\)-algebras
are AS-regular.
\end{enumerate}
\end{definition}

The main result of this paper is
\pref{th:birational},
which shows that
the moduli stack $\cM_Q$
admits a natural birational map
from the moduli stack
$\cMell$
of tuples
$(E, L_0, L_1, M_1, \ldots, M_n)$
defined in \pref{df:tuple}.
\pref{th:birational}
for noncommutative projective planes
follows from the classification
of quadratic AS-regular ($\bZ$-)algebras
in \cite{MR917738,MR1086882,MR1230966}.
\cite[Theorem 4.31]{MR2836401}
gives a strengthening of
\pref{th:birational}
for noncommutative quadrics.
\cite[Theorem 1.1]{1903.06457}
gives a variant of
\pref{th:birational} for noncommutative $\bP^2$ blown-up at one point.
\pref{th:birational} is a kind of reconstruction theorem,
which is proved using mutations of exceptional collections.
It is natural to expect
that the birational inverse
$
\cM_Q \dashrightarrow \Mell
$
is given
through the moduli scheme of \emph{point representations} of
the directed subquiver
\(
\Qdir \subset Q
\)
just as in \cite[Theorem 1.1]{2404.00175}.

If the very strong helix consists of line bundles,
then the moduli space of point representations
is cut out from a toric variety
by the relations of the quiver.
A del Pezzo surface $X$ of degree $d$
has a very strong helix of line bundles
only if $d \ge 3$
\cite[Theorem 5.13]{MR2822868}.
For each $d \ge 3$,
there exists a toric weak del Pezzo surface of degree $d$,
which deforms to a del Pezzo surface of the same degree.
Any toric weak del Pezzo surface
has a very strong collection of line bundles,
which deforms to a very strong collection of line bundles
on a del Pezzo surface of the same degree
(cf., e.g., \cite{Ishii-Ueda_DMEC}).

While Artin--Schelter surfaces of del Pezzo types
appearing in this paper are noncommutative blow-ups
of noncommutative projective planes
in the sense of \cite{MR1846352},
there are other approaches to noncommutative deformations of del Pezzo surfaces,
such as \cite{MR2734346}.
While \cite{MR2734346} is based on the notion of Calabi--Yau algebras,
which is closely related to the notion of AS-regular algebras
on which this paper is based,
there are marked differences,
such as that \cite{MR2734346} deals mostly with
noncommutative deformations of
affine del Pezzo surfaces,
and that \cite{MR2734346} studies deformation problems
rather than moduli problems.
In a sense, \cite{MR2734346} is local
both in the base and the fiber of the family
on the moduli space,
whereas our approach is global
in both directions.
An expectation on the relation
between \cite{MR1846352} and \cite{MR2734346}
is mentioned in \cite[Section 1.3]{MR2734346},
and it is an interesting problem
to pursue it.

\subsection*{Notation and Conventions}
Throughout this paper
we work over an algebraically closed field $\bfk$,
and assume that all categories and functors are linear over $\bfk$.
We follow the Grothendieck's convention
for the affine space and the projective space
so that
$
\bA V \coloneqq \Spec \Sym V
$
and
$
\bP V = \Proj \Sym V
$
for a $\bfk$-vector space $V$.
The dual space of a $\bfk$-vector $V$
is denoted by $V^\dual$.
Modules are right modules
unless otherwise specified.
For a pair $(X,Y)$ of objects
in a dg category,
the dg vector space of morphisms
and its zero-th cohomology group
is denoted by
$\hom(X,Y)$
and
$\Hom(X,Y)$
respectively.
We also write
$
\Hom^i(X,Y)
\coloneqq
\Hom(X,Y[i])
$
for $i \in \bZ$.

\subsection*{Acknowledgement}
S.~O.~was partially supported by
JSPS Grants-in-Aid for Scientific Research
No.19KK0348,
20H01797,
20H01794,
21H04994,
23K25771,
25K00907,
and the National Science Foundation under Grant No.~DMS-1928930,
while his residence at the Simons Laufer Mathematical Sciences Institute in Berkeley during the Spring 2024 semester.
K.~U.~was partially supported by
JSPS Grant-in-Aid for Scientific Research
No.21K18575.

\section{AS-regular $\bZ$-algebras}\label{sc:AS-regular Z-algebras}


A \emph{$\bZ$-algebra} is a category
with $\bZ$ as the set of objects.
A $\bZ$-algebra can be regarded as a collection
consisting of vector spaces $A_{ji} = \Hom_A(i,j)$,
the identity elements $\bfe_i \in A_{ii}$,
and the composition maps
$
A_{kj} \otimes A_{ji} \to A_{ki}
$
satisfying the associativity and the unit conditions.


A $\bZ$-algebra $A$ is said to be \emph{positively graded}
if $A_{ji} = 0$ for any $j < i$.
A positively graded $\bZ$-algebra is said to be \emph{connected}
if $A_{ii} = \bfk \bfe_i$ for any $i \in \bZ$.


A $\bZ$-algebra $A$ is said to have \emph{polynomial growth}
if there exits a polynomial
\(
   p ( t )
   \in
   \bQ [ t ]
\)
such that
\(
   \dim
   A
   _{
    i j
   }
   <
   p ( i - j )
\)
for any pair
\(
   ( i, j )
   \in
   \bZ ^{ 2 }
\).


The category
$\Module A$
of presheaves over a $\bZ$-algebra $A$ is called the category of right modules over $A$. Namely, a module over a $\bZ$-algebra $A$
is a contravariant functor from $A$ to the category of vector spaces.
It can be regarded as a collection
consisting of
vector spaces $M_i$
and maps $M_j \otimes A_{ji} \to M_i$
satisfying the associativity and the unit conditions.
Likewise, a homomorphism of modules is nothing but a natural transform.
The category $\Module A$ is a Grothendieck category.
The full subcategory of $\Module A$
consisting of Noetherian objects
is denoted by $\module A$.


For a connected $\bZ$-algebra $A$
and an integer $i$,
the
\(
   i
\)-the projective module (resp.~simple module)
\(
   P _{ i }
\)
(resp.~\(
   S _{ i }
\))
is the module represented by the object
\(
   i
\)
(resp. its
\(
   1
\)-dimensional quotient); concretely, $P_i = \bfe_i A$
(resp.~$S_i = \bfe_i A \bfe_i$).
One has
$
(P_i)_j = A_{ij}
$
(resp.~$(S_i)_i = \bfk$ and $(S_i)_j = 0$
for any $j \ne i$).


For a non-negative integer $d$,
a connected $\bZ$-algebra $A$ is said to be
\emph{AS-Gorenstein}
of dimension $d$
and parameter $\ell$
if for any
$
i \in \bZ
$,
one has
\begin{align}
    \Ext^k(S_i, P_j) \cong
    \begin{cases}
    \bfk & j = i-\ell \text{ and } k = d, \\
    0 & \text{otherwise},
    \end{cases}
\end{align}
where `AS' stands for Artin--Schelter
\cite{MR917738}.
An AS-Gorenstein $\bZ$-algebra is said to be \emph{AS-regular}
if
$A$ has polynomial growth
and
$\Module A$ has finite global dimension.


A module
\(
   M \in \Module A
\)
over a
\(
   \bZ
\)-algebra $A$ is left bounded if
\(
   M _{ i }
   =
   0
\)
for
\(
   i
    \ll
    0
\),
and is said to be a \emph{torsion module} if it is a colimit of left bounded modules. We let
\(
   \Tor A
   \subseteq
   \Module A
\)
denote the full subcategory of torsion modules. The quotient of $\module A$
by the Serre subcategory
\(
   \tor A
   \coloneqq
   \Tor A
   \cap
   \module A
\)
is denoted by $\qgr A$.

\section{Exceptional collections and helices}
 \label{sc:Exceptional collection and helices}


A dg category is \emph{proper}
if
$\hom(X,Y)$ is a perfect dg vector space
(i.e., quasi-isomorphic to a bounded complex
of finite-dimensional vector spaces)
for any objects $X, Y$.


An object $E$ of a proper dg category $\scrD$ is \emph{exceptional}
if
\begin{align}
 \hom(E, E) \simeq \bfk.
\end{align}


A sequence
$
(E_1, \ldots, E_\ell)
$
of exceptional objects
is an \emph{exceptional collection} if
\begin{align}
 \hom(E_j, E_i) \simeq 0
\end{align}
for any $1 \le i < j \le \ell$.


An exceptional collection
$
(E_1, \ldots, E_\ell)
$
is \emph{strong} if
\begin{align}
 \Hom^k(E_i, E_j) = 0
\end{align}
for any $k \ne 0$
and any $i, j \in \{ 1, \ldots, \ell \}$.


An exceptional collection
$
(E_1, \ldots, E_\ell)
$
is \emph{full}
if $\scrD$ is equivalent to
the pretriangulated hull
(i.e., the smallest full subcategory
closed under shifts and cones)
of $\{ E_1, \ldots, E_\ell \}$.


If
$
(E_1, \ldots, E_\ell)
$
is a full strong exceptional collection in $\scrD$,
then one has a quasi-equivalence
\begin{align}
 \hom_{\scrD} \lb \bigoplus_{i=1}^{\ell} E_i , - \rb
 \colon
 \scrD \simto D^b \module A
\end{align}
of dg categories
between $\scrD$ and the dg category of bounded complexes
of finitely generated right modules
over the total morphism algebra
$
 A = \bigoplus_{i,j=1}^{ \ell } \Hom(E_i, E_j)
$
\cite{MR992977,MR1002456,MR1258406}.


For $1 \le i \le \ell-1$,
the \emph{left mutation}
of an exceptional collection
$
 \tau = \lb E_j \rb_{j=1}^\ell
$
at position $i$
is defined by
\begin{align}
 L_i (\tau) \coloneqq
  \lb E_1, \ldots, E_{i-1}, L_{E_i} E_{i+1}, E_i, E_{i+2}, \ldots, E_\ell \rb
\end{align}
where
\begin{align}
 L_{E_i} E_{i+1} \coloneqq \Cone \lb \hom(E_i, E_{i+1}) \otimes E_i \xto{\ev} E_{i+1} \rb.
\end{align}
It is inverse
to the \emph{right mutation}
defined by
\begin{align}
 R_i (\tau) \coloneqq
  ( E_1, \ldots, E_{i-1}, E_{i+1}, R_{E_{i+1}} E_i, E_{i+2}, \ldots, E_\ell )
\end{align}
where
\begin{align}
  R_{E_{i+1}} E_i \coloneqq \Cone \lb E_i \xto{\coev} \hom(E_i, E_j)^\dual \otimes E_j \rb.
\end{align}


The \emph{helix}
generated by an exceptional collection $\lb E_i \rb_{i=1}^\ell$
is the sequence $\lb E_i \rb_{i \in \bZ}$
of exceptional objects
satisfying
\begin{align} \label{eq:helix}
 E_{i-\ell} = L_{E_{i-\ell+1}} \circ L_{E_{i-\ell+2}}
  \circ \cdots \circ L_{E_{i-1}}(E_i)[ - 2 ]
\end{align}
for any $i \in \bZ$.
The shift
\(
   [ - 2 ]
\)
in~\eqref{eq:helix} is chosen in such a way that
$E_{i-\ell} \cong \bS(E_i)[-2]$,
so that if
$\scrD \simeq D^b \coh X$
for a smooth projective surface $X$ and the collection is full, then one has
$E_{i-\ell} \cong E_i \otimes \omega_X$.
%
%
The length $\ell$ of the exceptional collection
generating the helix is called the \emph{period} of the helix.
%
%
For any $i \in \bZ$,
the exceptional collection $\lb E_j \rb_{j=i}^{i+\ell-1}$
consisting of $\ell$ consecutive objects in a helix
is called a \emph{foundation}
of the helix.
%
%
In this paper,
we will always assume
that a foundation of a helix is full.
%
%
A helix
$
\lb E_i \rb_{i \in \bZ}
$
is said to be \emph{strong}
if every foundation is strong.

\begin{definition} \label{df:very strong helix}
A helix 
$
\lb E_i \rb_{i \in \bZ}
$
is said to be \emph{very strong}
if
$
 \Hom^{k}(E_i, E_j) = 0
$
for any $-\infty < i < j < \infty$
and any $k \ne 0$.
\end{definition}
Very strongness is closely related to,
but different from,
the \emph{geometricity}
in \cite{MR1230966}.
%
%
As explained in \cite{2007.07620}, a very strong helix
$(E_i)_{i \in \bZ}$
of period $\ell$
gives the $\bZ$-algebra $A$
defined by
$
A_{ji} = \Hom(E_i, E_j),
$
which is AS-regular
with parameter $\ell$.
%
%
It follows from 
(the proof of) \cite[Theorem 16.(i)]{MR2641200}
that
if $A$ is an AS-regular $\bZ$-algebra
with parameter $\ell$,
then the sequence 
$(P_i)_{i \in \bZ}$
of projective modules
regarded as objects of $\qgr A$
gives a very strong helix
of period $\ell$,
and one has
\begin{align}
\Hom_{\qgr A}(P_i, P_j)
\cong
\Hom_{\module A}(P_i, P_j)
\end{align}
for any $i \le j$.

\section{Quivers and potentials}
\label{sc:quiver}


A \emph{quiver}
$
Q = (Q_0, Q_1, s, t)
$
consists of
\begin{itemize}
 \item
a set $Q_0$ of \emph{vertices},
 \item
a set $Q_1$ of \emph{arrows}, and
 \item
two maps
$s$ and $t$
from $Q_1$ to $Q_0$.
\end{itemize}
For an arrow $a \in Q_1$,
the vertices $s(a)$ and $t(a)$
are
called the \emph{source}
and the \emph{target}
of the arrow $a$
respectively.


A \emph{path}
is either
\begin{itemize}
\item
a path of length $k$
for some positive integer $k$,
which is a sequence
$p = (a_k, a_{k-1}, \dots, a_{1}) \in (Q_1)^k$
of arrows
such that $s(a_{i+1}) = t(a_i)$
for all $i \in \{ 1, \dots, k-1 \}$,
in which case we define the \emph{source} and the \emph{target} of $p$
by $s(p) \coloneqq s(a_1)$
and $t(p) \coloneqq t(a_k)$,
or 
\item
a path $\bfe_v$ of length zero
for some $v \in Q_0$,
in which case we set
$
s(\bfe_v)
= t(\bfe_v)
= v
$.
\end{itemize}


The \emph{concatenation} of paths
is defined by
\begin{align}
 (b_m, \dots, b_1) \cdot (a_n, \dots, a_1)
  = \begin{cases}
     (b_m, \dots, b_1, a_n, \dots, a_1) & s(b_1) = t(a_n), \\
      0 & \text{otherwise}
    \end{cases}
\end{align}
for paths of positive length,
and
\begin{align}
p \cdot \bfe_v
&=
\begin{cases}
p & s(p) = v, \\
0 & \text{otherwise},
\end{cases}
\\
\bfe_v \cdot p
&=
\begin{cases}
p & t(p) = v, \\
0 & \text{otherwise}
\end{cases}
\end{align}
if the length of at least one of the paths is zero.


For a non-negative integer $k$,
let
$\scrP_k(Q)$
be the set of paths of length $k$.
The \emph{path algebra} $\bfk Q$
is the algebra
freely generated
by the set
$
\scrP(Q)
\coloneqq
\coprod_{k=0}^\infty \scrP_k(Q)
$
of paths
as a vector space,
and the multiplication is defined
by the concatenation of paths.
It is graded by the length of the path as
\begin{align}
\bfk Q = \bigoplus_{k=0}^\infty \lb \bfk Q \rb_k,
\end{align}
where
$
\lb \bfk Q \rb_k
\coloneqq \bfk \scrP_k(Q)
$
is the vector space
freely generated by $\scrP_k(Q)$.


A \emph{quiver with relations} is a pair
$
\Gamma = (Q, I)
$
consisting of a quiver $Q$
and a two-sided ideal $I$
of the path algebra $\bfk Q$.
The \emph{path algebra} of $\Gamma$ is defined as
\begin{align}
\bfk \Gamma \coloneqq \bfk Q / I.
\end{align}


A \emph{cycle} or \emph{cyclic path} is a path
\(
   p
\)
satisfying the extra condition
\(
    s(p) = t(p)
\).
For a non-negative integer $k$,
let $\scrC_k(Q) \subset \scrP _{ k } ( Q )$ be the set of cyclic paths of length $k$,
and $\bfk \scrC_k(Q)$ the vector space
freely generated by $\scrC_k(Q)$.
The cyclic group of order $k$ acts on
\(
   \scrC_k(Q)
\)
as
$
(a_k,a_{k-1},\ldots,a_1)
\mapsto
(a_1,a_k,\ldots,a_2)
$.

For each $a \in Q_1$,
we define a map
\begin{align}
\frac{\partial}{\partial a}
\colon
\bfk \scrC_k(Q) /
\left(
    \bZ / k \bZ
\right)
\to
\lb \bfk Q \rb_{k-1}
\end{align}
by sending the orbit of the cyclic path
\(
   p
   =
    (a_k,\ldots,a_1)
\)
to
\begin{align}
  \sum_{i=1}^k \delta_{a_i,a} a_{i-1} \cdots a_1 a_k \cdots a_{i+1}
\end{align}
and extending linearly,
where
\begin{align}
  \delta_{a,a'} =
\begin{cases}
  1 & a = a', \\
  0 & \text{otherwise}
\end{cases}
\end{align}
is the Kronecker delta.
A \emph{quiver with potential}
is a pair $(Q, \Phi)$
consisting of a quiver $Q$
and an element
$
\Phi \in
\bigoplus_{k=0}^\infty
\bfk \scrC_k(Q) /
\left(
    \bZ / k \bZ
\right)
$
called a \emph{potential}.
A quiver with potential $(Q, \Phi)$
gives rise to a quiver with relations $\Gamma = (Q, \partial \Phi)$
where
\begin{align}
    \partial \Phi \coloneqq \lb \frac{\partial}{\partial a} \Phi \rb_{a \in Q_1}.
\end{align}

A cyclic path is said to be \emph{primitive}
if it is not a concatenation of two or more cyclic paths.
The set of primitive cyclic paths is denoted by $\scrC^{\mathrm{prim}}(Q)$.

\begin{definition} \label{df:moduli stack of potentials}
The \emph{moduli stack of potentials} of a quiver $Q$
is the quotient stack
\begin{align}
\cM_Q
\coloneqq
\ld
\bA^{\scrC^{\mathrm{prim}}(Q)}
\middle/
\Aut \lb \bfk Q \middle/ \bfk^{Q_0} \rb
\rd
\end{align}
of the affine space
generated by the set of primitive cyclic paths
by
the natural action
of the group
$
\Aut \lb \bfk Q \middle/ \bfk^{Q_0} \rb
$
of automorphisms of $\bfk Q$
as an algebra over the semisimple algebra $\bfk^{Q_0}$.
\end{definition}

\section{Rolled-up helix quiver}

Let $(E_i)_{i \in \bZ}$ be a very strong helix
on a del Pezzo surface $X$
of degree $d$,
so that the associated $\bZ$-algebra $A$ is
AS-regular of dimension 3 with parameter $\ell = 12-d$.
The \emph{rolled-up helix algebra}
is
defined as
\begin{align}\label{eq:rolled-up helix algebra}
\End \lb \bigoplus_{i=1}^{\ell} \pi^* E_i \rb
\cong
\bigoplus_{i=1}^\ell \bigoplus_{k=0}^\infty
\Hom \lb E_i, E_{i+k} \rb
\end{align}
where
$
\pi \colon \bA \lb \omega_X^\dual \rb
\to X
$
is the total space of the canonical bundle of \( X \).
It is a Calabi--Yau algebra of dimension 3
in the sense of
\cite[Definition~3.2.3]{0612139},
so that it comes from a quiver $(Q,\Phi)$ with potential.
The quiver
\(
   Q
\),
which depends only on the deformation class of the del Pezzo surface $X$,
will be called the \emph{rolled-up quiver}.
The set of vertices of $Q$ is given by
\begin{align}
  Q_0 = \{ 1, \ldots, \ell \},
\end{align}
which is identified with the foundation
$
(E_i)_{i=1}^\ell
$
of the helix.
For $1 \le i, j \le \ell$,
the set of arrows from $i$ to $j$
is identified with a basis
$
\{ x_{jik} \}_{k=1}^{m_{ji}}
$
of the space of \emph{primitive morphisms}
\begin{align}\label{eq:primitive morphisms}
\Hom^{\mathrm{prim}}(E_i,E_j)
\coloneqq
\left.
\Hom(E_i, E_j)
\middle/
\Image
\left(
    \bigoplus_{i < k < j} \Hom(E_k,E_j) \otimes \Hom(E_i,E_k)
    \to
    \Hom
    \left(
        E _{ i },
        E _{ j }
    \right)
\right)
\right.
\end{align}
if $i \le j$
(resp.,
$
\Hom^{\mathrm{prim}}(E_i,E_{j+\ell})
$
if
$i > j$).

\begin{proposition}
A resolution of the simple $A$-module $S_i$ is given by
\begin{align} \label{eq:resolution}
0
\to
P_{i-\ell}
\to
\bigoplus_{i < j \le \ell} P_{j-\ell}^{\oplus m_{ji}}
\to
\bigoplus_{1 \le j < i} P_j^{\oplus m_{ij}}
\to
P_{i}
\to
S_{i}
\to
0.
\end{align}
\end{proposition}

\begin{proof}
Since the rolled-up helix algebra~\eqref{eq:rolled-up helix algebra} is graded Calabi-Yau of dimension 3, where the grading comes from the index
\(
   k
\)
in~\eqref{eq:rolled-up helix algebra}, the diagonal bimodule and hence the simple modules have a particular type of resolutions whose shape depends only on the quiver~\( Q \) (\cite[Theorem~4.5]{MR2355031}). On the other hand, the
\(
   \bZ
\)-algebra associated to~\eqref{eq:rolled-up helix algebra} is naturally isomorphic to
\(
   A
\), through which we identify graded right modules over~\eqref{eq:rolled-up helix algebra} with right modules over the $\bZ$-algebra \( A \) (cf.,~\cite[equation~(20)]{2007.07620}). Thus we obtain the resolutions~\eqref{eq:resolution}.
\end{proof}

Let us fix a quiver
\(
   Q
\)
as above.

\begin{definition} \label{df:AS-regular of type Q}
An AS-regular $\bZ$-algebra
is said to be \emph{of type $Q$}
if a resolution of the simple module $S_i$
has the form \eqref{eq:resolution}
for any $i \in \bZ$.
\end{definition}

\begin{example}
    Let
    \(
       Q
    \)
    be the rolled-up quiver obtained from the helix
    \(
    \left(
        \cO
       _{
        \bP ^{ 2 }
       }
       ( i )
    \right)
    _{
        i \in \bZ
    }
    \)
    on
    \(
         \bP ^{ 2 }
    \). Then AS-regular 
    \(
       \bZ
    \)-algebras of type
    \(
       Q
    \)
    are precisely the 3-dimensional AS-regular quadratic
    \(
       \bZ
    \)-algebras. Similarly, from the helix
    \(
       \left(
        \dots,
        \cO ( - 1, 0 ),
        \cO ( 0, 0 ),
        \cO ( 0, 1 ),
        \cO ( 1, 1 ),
        \dots
       \right)
    \)
    on
    \(
       \bP ^{ 1 } \times \bP ^{ 1 }
    \)
    we obtain the definition of 3-dimensional AS-regular cubic
    \(
       \bZ
    \)-algebras (see~\cite{MR2836401}). A class of AS-regular
    \(
       \bZ
    \)-algebras obtained from a 3-block helix on cubic surfaces is discussed in~\cite[Section~4]{2404.00175}.
\end{example}

\begin{remark}
While the set $\bZ$
in \pref{df:AS-regular of type Q}
comes from the indexing set of the helix
$\lb E_i \rb_{i \in \bZ}$,
the use of $\bZ$
as an indexing set of a helix is
often
not canonical,
and
it is more natural to talk
about AS-regular $\Qv$-algebras
than AS-regular $\bZ$-algebras of type $Q$,
where
\(
   \Qv
   =
   \left(
    \Qv_0, \Qv_1, \sv, \tv
   \right)
\)
is the quiver
defined by
\begin{align}
\Qv_0
&= Q_0 \times \bZ, \\
\Qv_1
&= Q_1 \times \bZ, \\
\sv(a,j)
&= (s(a),j), \\
\tv(a,j)
&=
\begin{cases}
(t(a),j) & s(a) < t(a), \\
(t(a),j+1) & \text{otherwise}.
\end{cases}
\end{align}
The set $\Qv_0$ has an order $\prec$
which is initial among those satisfying
the condition
that $\sv(a) \prec \tv(a)$ for any $a \in \Qv_1$,
and
an identification of $\Qv_0$
with $\bZ$
is given by
a choice of an order-preserving bijection
$
(\Qv_0, \prec) \to (\bZ, <)
$.
In terms of $\Qv$,
the resolution
\pref{eq:resolution}
can be written as
\begin{align}
    0
    \to
    P
    _{
    \left(
        i,
        d - 1        
    \right)            
    }
    \to
    \bigoplus
    _{
        \substack{
            \bv
            \in
            \Qv _{ 1 }
            \\
            \sv ( \bv )
            =
            \left(
                i,
                d - 1        
            \right)
        }
    }
    P
    _{
        \tv ( \bv )
    }
    \to
    \bigoplus
    _{
        \substack{
            \av
            \in
            \Qv _{ 1 }
            \\
            \tv ( \av )
            =
            \left(
                i,
                d        
            \right)
        }
    }
    P
    _{
        \sv ( \av )
    }
    \to
    P
    _{
    \left(
        i,
        d        
    \right)
    }
    \to
    S        
    _{
        \left(
            i,
            d        
        \right)
    }
    \to
    0.
\end{align}
\end{remark}

\section{Marked Artin--Schelter surfaces and geometric data}\label{sc:Marked Artin--Schelter surfaces and geometric data}

One of the most significant results
concerning 3-dimensional AS-regular quadratic or cubic $\bZ$-algebras
is the bijective correspondence
between such algebras and certain geometric data
consisting of a genus one curve with additional structure
\cite{MR1086882,MR1230966,MR2836401}.
In this section,
we partially generalize this correspondence
to AS-regular \(\bZ\)-algebras
associated with an arbitrary del Pezzo surface,
formulating it as a birational map
between the moduli stacks.

\begin{definition}\label{df:tuple}
    The stack $\cMell$ is the moduli stack of tuples
    \(
       \left(
        E,
        L_0,
        L_1,
        M _{ 1 },
        \ldots,
        M _{ n }
       \right)
    \)
    consisting of
    \begin{itemize}
        \item 
        a smooth projective curve
         \(
           E
        \)
        of genus 1,
        \item
         two line bundles
        \(
           L_0
        \)
        and
        \(
           L_1
        \)
        on
        \(
           E
        \)
        of degree
        \(
           3
        \),
         and
        \item
         line bundles
        \(
           M_1, \ldots, M_n
        \)
        on
        \(
           E
        \)
        of degree
        \(
           1
        \).
    \end{itemize}
    An isomorphism from
    \(
       (E, L_0, L_1, M_1, \ldots, M_n)
    \)
    to
    \(
       (E', L_0', L_1', M_1', \ldots, M_n')
    \)
    is a tuple
    \(
       ( f, \varphi_0, \varphi_1, \psi _{ 1 }, \dots, \psi _{ n } )
    \)
    consisting of
    \begin{itemize}
        \item
         an isomorphism
        \(
           f \colon E \to E'
        \)
        of curves
         and
        \item
         isomorphisms
        \(
           \varphi _{ j }
           \colon
            L _{ j }
            \simto
            f ^{ \ast } L _{ j }'
        \)
         and
        \(
           \psi _{ k }
           \colon
            M _{ k }
            \simto
            f ^{ \ast } M _{ k }'
        \)
        of line bundles for
        \(
           j = 0, 1
        \)
        and
        \(
           k = 1, \ldots, n
        \).
    \end{itemize}
\end{definition}

Fix a very strong helix
on a del Pezzo surface
and an element
\(
\sigma
\)
of the braid group
\(
\Br_{n+3}
\)
on
\(
   n + 3
\)
strands such that the sequence of mutations given by
\(
   \sigma
\)
turns the Beilinson--Orlov collection
\(
(
\cO,
\cO ( 1 ),
\cO ( 2 ),
\cO _{ E _{ 1 } },
\dots,
\cO _{ E _{ n } }
)
\)
into the foundation of the very strong helix,
whose existence is guaranteed by~\cite[Theorem~7.7]{MR1286839}.
The corresponding rolled-up quiver is denoted by $Q$.

Given a tuple
\(
    \left(
    E,
    L_0,
    L_1,
    M _{ 1 },
    \ldots,
    M _{ n }
    \right)
\)
as in~\pref{df:tuple},
let
\(
\left(
\cS _{ i }
\right)_{i=1}^{n+3}
\)
be the spherical collection
obtained from
\(
   \left(
    \cO _{ E },
    L _{ 0 },
    L _{ 0 } \otimes L _{ 1 },
    M _{ 1 },
    \dots,
    M _{ n }
   \right)
\)
by the sequence of spherical twists
specified by the element
\(
\sigma \in \Br_{n+3}
\).

\begin{lemma}\label{lm:very strong}
If the tuple
\(
    \left(
    E,
    L_0,
    L_1,
    M _{ 1 },
    \ldots,
    M _{ n }
    \right)
\)
is sufficiently general,
then the spherical helix
generated by
\(
\left(
\cS_i
\right)_{i=1}^{n+3}
\)
is very strong.
\end{lemma}

\begin{proof}
Let
\(
   E
   \hookrightarrow
   \Bl _{ p _{ 1 }, \ldots, p _{ n } }
   \bP
   H ^{ 0 }
   \left(
   L _{ 0 }
   \right)
\)
be the embedding of the curve $E$
to the blow-up of
\(
   \bP H ^{ 0 } \left(
   L _{ 0 }
   \right)
   \simeq
   \bP ^{ 2 }
\)
at the points
\(
   p _{ 1 }, \ldots, p _{ n }
\)
such that
\(
   M _{ k }
   \simeq
   \cO _{ E } ( p _{ k } )
\)
for
\(
   k = 1, \ldots, n
\),
which lifts
\(
   E
   \hookrightarrow
   \bP H ^{ 0 } \left(
   L _{ 0 }
   \right)
\).
If
\(
   L _{ 0 } \simeq L _{ 1 }
\), then the spherical helix on
\(
   E
\)
is obtained by restricting the very strong helix on the surface,
in which case it is automatically very strong.

If
\(
   L _{ 0 }
   \not\simeq
   L _{ 1 }
\),
then
there are only finitely many cohomology vanishings to be checked for a spherical collection on
\(
   E
\)
to be very strong. This follows from the Riemann--Roch and that spherical objects on a smooth projective curve of genus 1 are either semistable vector bundles or skyscraper sheaves up to shift (see, e.g.,~\cite{MR2264663}). Deforming from the case where
\(
   L _{ 0 }
   \simeq
   L _{ 1 }
\),
where the spherical collection is very strong as we confirmed in the previous paragraph,
we obtain the conclusion
by the upper semi-continuity of the dimensions of cohomology groups.
\end{proof}

\begin{definition}
Let
\begin{align}\label{eq:potential to geometric tuple}
   \scrA _{ Q } \colon \cMell \dashrightarrow \cM_Q
\end{align}
be the rational map of stacks
sending a general tuple
\(
\left(
E,
L_0,
L_1,
M _{ 1 },
\ldots,
M _{ n }
\right)
\)
to the potential of the AS-regular $\bZ$-algebra of type $Q$
associated
by \cite[Theorem~5]{2007.07620}
with the very strong spherical helix
generated by the spherical collection
\(
\left(
\cS_i
\right)_{i=1}^{n+3}
\).
An isomorphism of tuples as in~\pref{df:tuple} induces an isomorphism of the dg categories whose objects are the corresponding spherical collections, which hence induces an isomorphism of the corresponding potentials.
\end{definition}

Now we construct a rational map
\begin{align}\label{eq:geometric tuple to potential}
     \scrB _{ Q }
     \colon
     \cM_Q
     \dashrightarrow
     \cMell,
\end{align}
which we will show to be the inverse of
\(
   \scrA_Q
\).
Given a general potential of the quiver $Q$,
let $A$ be the corresponding AS-regular $\bZ$-algebra of type
\(
   Q
\).
Then the sequence
$
    (P_i)_{i=1}^\ell
$
where $P_i$ is the $i$-th projective $A$-module
regarded as an object of $\qgr A$
is a foundation of a very strong helix.
We introduce the notation
\begin{align}\label{eq:Beilinson--Orlov type collection}
(\cO_{l_1}(-1),\ldots,\cO_{l_n}(-1),\cO_X,\cO_X(l),\cO_X(2l))
\end{align}
for the objects
constituting the exceptional collection
obtained from $(P_i)_{i=1}^\ell$
by the sequence of mutations specified by
\(
   \sigma ^{ - 1 }
\).
The semiorthogonal summand
$
\la \cO_X,\cO_X(l),\cO_X(2l) \ra
$
can be identified
with the derived category $D^b \coh X_0$
of a noncommutative projective plane $X_0$. Indeed, by the upper semi-continuity of the dimensions of cohomology groups, the genericity of the potential implies that
\(
   \cO _{ X },
    \cO _{ X } ( l ),
    \cO _{ X } ( 2 l )
    \in
    \langle
        \cO _{ X },
        \cO _{ X } ( l ),
        \cO _{ X } ( 2 l )
    \rangle
\)
is a full strong exceptional collection whose endomorphism algebra is isomorphic to the quotient of the path algebra of the Beilinson quiver with three vertices by three relations of length two which are sufficiently general in the moduli, so that they coincides with relations corresponding to a 3-dimensional Sklyanin algebra.
Let $(E,L_0,L_1)$ be the triple
describing the isomorphism class of $X_0$. Note that it is functorially obtained from the quiver with relations as the moduli space of point representations and (the data obtained from) the universal representation on the moduli space, and that there is a canonical functor
\(
   D ^{ b } \coh E
   \to
    D ^{ b } \coh X_0
    =
    \la \cO _{ X }, \cO _{ X } ( l ), \cO _{ X } ( 2 l ) \ra
\).
Again by the upper semi-continuity of the dimensions of cohomology groups, the genericity of the potential implies that the semiorthogonal summands
$
\la \cO_{l_k}(-1) \ra
$
for $1 \le k \le n$
are mutually orthogonal,
and
the gluing functors
$
\la \cO_{l_k}(-1) \ra
\to
D^b \coh X_0
$
are determined
by their images,
which we write as
$
\cO_{p_k}
$.
The object $\cO_{p_k}$ are point objects
(i.e., those coming from point modules of the corresponding
AS-regular ($\bZ$-)algebra)
at least if the potential of $A$ is sufficiently general
in $\cM_Q$,
which is the pushforward of the skyscraper sheaf
\(
   \cO _{ p _{ k } } \in \coh E
\)
along the canonical functor. Finally, we obtain the objects
\(
   M _{ k } = \cO _{ E } ( p _{ k } )
\)
as the twist of
\(
   \cO _{ E }
\)
by the spherical object
\(
    \cO _{ p _{ k } } [ - 1 ]
\).
This gives a construction of a tuple
\(
    \left(
    E,
    L_0,
    L_1,
    M _{ 1 },
    \ldots,
    M _{ n }
    \right)
\)
as in~\pref{df:tuple} from
\(
    A
\).
The definition of the action of the functor
\(
   \scrB _{ Q }
\)
on isomorphisms of potentials is rather straightforward, as all steps of the construction of the tuple are functorial.

This implies an equivalence
\(
   D ^{ b } \qgr A
   \simeq
    D ^{ b } \qgr S
\)
where the graded algebra
\(
   S
\)
is the noncommutative blow-up
of the noncommutative projective plane
\(
    X _{ 0 }
\)
at
\(
   p _{ 1 }, \ldots, p _{ n }
\)
in the sense of~\cite{MR1846352},
since
\(
   D ^{ b } \qgr S
\)
is
also obtained by gluing
\(
   n
\)
copies of the derived category of
\(
   \Spec \bfk
\)
to
\(
   D ^{ b } \coh X _{ 0 }
\)
along the objects
\(
   \cO _{ p _{ k } }
\)
\cite[Theorem~8.4.1]{MR1846352}.

There is an adjoint pair of functors
\begin{align}
    \iota ^{ \ast }
    \colon
    D ^{ b } \qgr S
    \rightleftarrows
    D ^{ b } \coh E
    \colon
    \iota _{ \ast }
\end{align}
by~\cite[Theorem~6.2.2]{MR1846352}.
When
\(
   X _{ 0 } \cong \bP ^{ 2 }
\),
this comes from the embedding
\(
   \iota
    \colon
    E
    \hookrightarrow
    \Bl _{ p _{ 1 }, \ldots, p _{ n } }
    X _{ 0 }
\)
of
\(
   E
\)
as an anti-canonical divisor.
In particular, the full subcategory
\(
   \left(
    P _{ i }
   \right)
   _{
    i = 1
   }
   ^{ \ell }
   \subset   
    D ^{ b } \qgr A
    \simeq
   D ^{ b } \qgr S
\)
coincides with the directed subcategory
of the dg category
\(
   \left(
    \iota ^{ \ast } P _{ i }
   \right)
      _{
    i = 1
   }
   ^{ \ell }
   \subset
    D ^{ b } \coh E
\)
(i.e., the non-full subcategory
with the same set of objects
obtained by throwing away morphisms
from $P_i$ to $P_j$ for $i > j$
and endomorphisms of all objects
not proportional to the identity morphisms).
By the upper semi-continuity of the dimensions of cohomology groups,
the same applies to the general case.
It follows that the cone
of the unit
\(
   \id
   \to
    \iota _{ \ast } \iota ^{ \ast }
\)
of the adjunction
is isomorphic to the Serre functor of
\(
   D ^{ b } \qgr S
\)
shifted by
\(
   [ - 1 ]
\),
so that the restriction functor
\(
   \iota ^{ \ast }
\)
commutes with mutations of exceptional collections in
\(
   D ^{ b } \qgr S
\)
and mutations of spherical collections in
\(
   D ^{ b } \coh E
\).


\begin{theorem} \label{th:birational}
The rational map $\scrA_Q$ of~\eqref{eq:potential to geometric tuple}
is inverse to the rational map $\scrB_Q$
of~\eqref{eq:geometric tuple to potential}.
\end{theorem}

\begin{proof}
The equivalence
\(
\scrB _{ Q }
\circ
\scrA _{ Q }
\simeq
\id
_{
\cMell
}
\)
is straightforward.
To show
\(
\scrA _{ Q }
\circ
\scrB _{ Q }
\simeq
\id
_{
   \cM_Q
}
\), take a general potential of the quiver
\(
Q
\)
and let
\(
I
\subset
\bfk \Qdir
\)
be the corresponding relations of the directed subquiver
\(
\Qdir
\)
(i.e., the quiver with the same set of vertices as $Q$
and the set of arrows consisting of all arrows
from a vertex to a vertex with a larger index).
The assertion to be checked is that
if we perform the sequence of mutations given by
\(
\sigma ^{ - 1 } \in \Br_{n+3}
\)
to the full exceptional collection
\(
\left(
P _{ i }
\right)
_{
i = 1
}
^{ \ell }
\)
of
\(
D^b \qgr A
\simeq
D^b \module \bfk \Qdir / I
\)
to obtain the Beilinson--Orlov type collection as in~\eqref{eq:Beilinson--Orlov type collection}, take the corresponding geometric tuple as above, take the associated spherical collection obtained by performing the spherical twists specified by
\(
\sigma \in \Br_{n+3}
\),
and then take the directed subcategory of thus-obtained spherical collection,
then it coincides with the original full strong exceptional collection
\(
\left(
P _{ i }
\right)
_{
i = 1
}
^{ \ell }
\).
This follows from the compatibility of mutations of exceptional collections
and spherical collections
discussed above.
\end{proof}

\bibliographystyle{amsalpha}
\bibliography{bibs}

\end{document}